\documentclass[a4,12pt]{amsart}
\usepackage{amssymb,amstext,amsmath,amscd,amsthm,
amsfonts,enumerate,graphicx,latexsym}
\usepackage{url}
\usepackage[usenames]{color}
\definecolor{gr}{rgb}{0.1, .5 , .10}
\usepackage[all]{xy}

\newtheorem{theorem}{Theorem}[section]
\newtheorem{theorem*}{Theorem}
\newtheorem{corollary}[theorem]{Corollary}
\newtheorem{corollary*}[theorem*]{Corollary}

\newtheorem{proposition}[theorem]{Proposition}

\theoremstyle{definition}
\newtheorem{definition}[theorem]{Definition}
\newtheorem{remark}[theorem]{Remark}

\newtheorem*{question*}{Question}

\newtheorem*{conjecture*}{Conjecture}
\newtheorem{example}[theorem]{Example}

\newtheorem*{notation*}{Notation}

\newtheorem*{claim*}{Claim}

\numberwithin{equation}{theorem}
\def\Ext{\operatorname{Ext}}

\def\rad{\operatorname{rad}}
\def\soc{\operatorname{soc}}
\def\Hom{\operatorname{Hom}}

\def\sCM{\operatorname{\underline{\mathsf{CM}}}}

\def\Db{\mathsf{D}^{\rm b}}

\def\thick{\operatorname{\mathsf{thick}}}
\def\add{\operatorname{\mathsf{add}}}

\def\mod{\operatorname{\mathsf{mod}}}
\def\smod{\operatorname{\underline{\mathsf{mod}}}}
\def\proj{\operatorname{\mathsf{proj}}}
\def\inj{\operatorname{\mathsf{inj}}}
\def\Db{\mathsf{D^b}}
\def\Kb{\mathsf{K^b}}
\def\ds{\mathsf{D_{sg}}}

\def\id{\mathrm{inj}.\!\dim}

\def\sttilt{\operatorname{\mathsf{s\tau -tilt}}}
\def\silt{\operatorname{\mathsf{silt}}}
\def\A{\mathcal{A}}
\def\C{\mathcal{C}}

\def\T{\mathcal{T}}

\def\H{\mathcal{H}}

\renewcommand{\S}{\mathcal{S}}

\newcommand{\op}{{\rm op}}

\begin{document}
\setlength{\baselineskip}{15pt}
\title{Report on the finiteness of silting objects}
\author{Takuma Aihara}
\address{Department of Mathematics, Tokyo Gakugei University, 4-1-1 Nukuikita-machi, Koganei, Tokyo 184-8501, Japan}
\email{aihara@u-gakugei.ac.jp}
\author{Takahiro Honma}
\address{Department of Mathematics, Tokyo University of science, 1-3 Kagurazaka, Shinjuku, Tokyo 162-8601, Japan}
\email{1119704@ed.tus.ac.jp}
\author{Kengo Miyamoto}
\address{Department of General Education, National Institute of Technology (KOSEN), Yuge College, Ochi, Ehime 794-2506, Japan}
\email{k\_miyamoto@yuge.ac.jp}
\author{Qi Wang}
\address{Department of Pure and Applied Mathematics, Graduate School of Information
Science and Technology, Osaka University, Suita, Osaka 565-0871, Japan}
\email{q.wang@ist.osaka-u.ac.jp}

\thanks{TA was partly supported by JSPS Grant-in-Aid for Young Scientists 19K14497.}
\thanks{KM was partly supported by JSPS KAKENHI 18J10561
and Association for Advancement of Technology, NIT (KOSEN), Yuge college.}
\begin{abstract}
We discuss the finiteness of (two-term) silting objects.
First, we investigate new triangulated categories without silting object.
Second, one studies two classes of $\tau$-tilting-finite algebras and give the numbers of their two-term silting objects.
Finally, we explore when $\tau$-tilting-finiteness implies representatoin-finiteness, and obtain several classes of algebras in which a $\tau$-tilting-finite algebra is representation-finite. 
\end{abstract}
\maketitle
\section{Introduction}

In this paper, we discuss three subjects on the finiteness of (two-term) silting objects.

The case we should first consider is when the number of silting objects is zero;
that is, the question is which triangulated categories have no silting object.
For example, the bounded derived category $\Db(\mod\Lambda)$ over a finite dimensional algebra $\Lambda$ has no non-zero silting object if and only if $\Lambda$ has infinite global dimension.
When $\Lambda$ is non-semisimple selfinjective, its stable module category $\smod\Lambda$ admits no silting object.
(See \cite{AI}.)
Inspired by these two cases, we ask if the singularity category $\ds(\Lambda)$ of $\Lambda$ has no non-zero silting object.
Here is the first main theorem of this paper.

\begin{theorem*}[Theorem \ref{ns} and Corollary \ref{IG}]
$\ds(\Lambda)$ admits no non-zero silting object if $\Lambda$ has finite right selfinjective dimension.
In particular, the stable category of the Cohen--Macaulay category over an Iwanaga--Gorenstein algebra has no non-zero silting object.
\end{theorem*}

Next, we restrict our viewpoint from silting objects to two-term silting ones.
Thanks to Adachi--Iyama--Reiten, there is a one-to-one correspondence between two-term silting objects and support $\tau$-tilting modules \cite{AIR}.
We call an algebra \emph{$\tau$-tilting-finite} provided there are only finitely many support $\tau$-tilting modules.
For example, we know the following $\tau$-tilting-finite algebras:
representation-finite algebras, preprojective algebras of Dynkin type \cite{M} and algebras of dihedral, semi-dihedral and quaternion type \cite{EJR} and so on.

The second subject of this paper is to give two new classes of $\tau$-tilting-finite algebras.
One is the class of weakly symmetric algebras of tubular type with non-singular Cartan matrix \cite{BS1, BHS}.
The other is the class of non-standard selfinjective algebras which are socle equivalent to selfinjective algebras of tubular type \cite{BS2, BHS}.
Here is the second main theorem of this paper.
(See Figure \ref{Lwst} and \ref{LLs} for the notation of $A_i$'s and $\Lambda_i$'s.)

\begin{theorem*}[Theorem \ref{finiteness1} and Theorem \ref{finiteness2}]
\begin{enumerate}
\item Any weakly symmetric algebra of tubular type with non-singular Cartan matrix is $\tau$-tilting-finite. In particular, we have the number of support $\tau$-tilting modules:
\[\def\arraystretch{1.5}
\begin{array}{|@{\hspace{7pt}}c@{\hspace{7pt}}|@{\hspace{7pt}}c@{\hspace{7pt}}|@{\hspace{7pt}}c@{\hspace{7pt}}|@{\hspace{7pt}}c@{\hspace{7pt}}|@{\hspace{7pt}}c@{\hspace{7pt}}|@{\hspace{7pt}}c@{\hspace{7pt}}|@{\hspace{7pt}}c@{\hspace{7pt}}|@{\hspace{7pt}}c@{\hspace{7pt}}|}\hline
A_1(\lambda) & A_2(\lambda) & A_3 & A_4 & A_5 & A_6 & A_7 & A_8 \\\hline
24&6&192&132&8&8&108& 100 \\\hline\hline
A_9 & A_{10} & A_{11} & A_{12} & A_{13} & A_{14} & A_{15} & A_{16}\\\hline
108 &116 & 100 & 32 & 28 & 32 & 30 & 30 \\\hline
\end{array}
\]

\item Any non-standard selfinjective algebra which is socle equivalent to a selfinjective algebra of tubular type is $\tau$-tilting-finite.
In particular, we have the number of support $\tau$-tilting modules:
\[\def\arraystretch{1.5}
\begin{array}{|@{\hspace{7pt}}c@{\hspace{7pt}}|@{\hspace{7pt}}c@{\hspace{7pt}}|@{\hspace{7pt}}c@{\hspace{7pt}}|@{\hspace{7pt}}c@{\hspace{7pt}}|@{\hspace{7pt}}c@{\hspace{7pt}}|@{\hspace{7pt}}c@{\hspace{7pt}}|@{\hspace{7pt}}c@{\hspace{7pt}}|@{\hspace{7pt}}c@{\hspace{7pt}}|@{\hspace{7pt}}c@{\hspace{7pt}}|@{\hspace{7pt}}c@{\hspace{7pt}}|}\hline
\Lambda_1 & \Lambda_2& \Lambda_3(\lambda)  & \Lambda_4 & \Lambda_5 & \Lambda_6 & \Lambda_7 & \Lambda_8 &\Lambda_9&\Lambda_{10} \\\hline
8&8&6&32&28&32&30&30&192&\geq500\\\hline
\end{array}
\]
\item Every algebra as in (1) and (2) is tilting-discrete.
\end{enumerate}
\end{theorem*}

The last subject is on ``representation-finiteness vs. $\tau$-tilting-finiteness''.
Evidently, a representation-finite algebra is $\tau$-tilting-finite, but the converse does not necessarily hold.
Thus, we naturally ask when $\tau$-tilting-finiteness implies representation-finiteness.
A typical example is the hereditary case; that is, $\tau$-tilting-finite hereditary algebras are representation-finite.
For more examples, it was proved that $\tau$-tilting-finite cycle-finite algebras are representation-finite \cite{MS}.
Recently, the gentle case was verified; $\tau$-tilting-finite gentle algebras are representation-finite \cite{P}.
Now, we give new classes of algebras which are the case.

\begin{theorem*}[Corollary \ref{Zito}, \ref{separation}, Theorem \ref{rs0} and \ref{lh}]
The following algebras are representation-finite if they are $\tau$-tilting-finite:
\begin{enumerate}
\item quasitilted algebras;
\item algebras satisfying the separation condition;
\item the trivial extensions of tree quiver algebras with radical square zero;
\item locally hereditary algebras;
\end{enumerate}
\end{theorem*}

\begin{notation*}
Throughout this paper, algebras are always assumed to be basic, indecomposable and finite dimensional over an algebraically closed field $K$.
Modules are finitely generated and right.
For an algebra $\Lambda$, we denote by $\mod\Lambda\ (\proj\Lambda, \inj\Lambda)$ the category of (projective, injective) modules over $\Lambda$.
\end{notation*}


\section{The existence of silting objects}

Let $\T$ be a Krull--Schmidt triangulated category which is $K$-linear and Hom-finite.
For example, we consider the bounded derived category $\Db(\mod\Lambda)$ and the perfect derived category $\Kb(\proj\Lambda)$ over an algebra $\Lambda$.
In this section, we explore when a triangulated category has no silting object.
Let us recall the definition of silting objects.

\begin{definition}
An object $T$ of $\T$ is said to be \emph{presilting} (\emph{pretilting}) if it satisfies $\Hom_\T(T, T[i])=0$ for any $i>0$ ($i\neq0$).
It is called \emph{silting} (\emph{tilting}) if in additional $\T=\thick T$.
Here, $\thick T$ stands for the smallest thick subcategory of $\T$ containing $T$.
We denote by $\silt\T$ the set of isomorphism classes of basic silting objects of $\T$.
\end{definition}

A typical example of silting objects is the stalk complex $\Lambda$ (and its shifts) in $\Kb(\proj\Lambda)$.
If we can find even one silting object, silting mutation produces infinitely many ones \cite{AI}.
However, we know triangulated categories with no silting object \cite[Example 2.5]{AI}.

Let $\Lambda$ be an algebra.
We denote by $\ds(\Lambda)$ the singularity category of $\Lambda$; that is, it is the Verdier quotient of $\Db(\mod\Lambda)$ by $\Kb(\proj\Lambda)$.
Here is the main result of this section.

\begin{theorem}\label{ns}
$\ds(\Lambda)$ has no non-zero silting object if $\id\Lambda_\Lambda<\infty$.
\end{theorem}

To prove this theorem, silting reduction \cite{AI, IY} plays a crucial role.

In the rest, fix a presilting object $T$ of $\T$ and  define a subset $\silt_T\T$ of $\silt\T$ by
\[\silt_T\T:=\{P\in\silt\T\ |\ T\ \mbox{is a direct summand of } P \}.\]
Moreover, one puts $\S:=\thick T$.
The Verdier quotient of $\T$ by $\S$ is denoted by $\T/\S$.

Then, silting reduction \cite[Theorem 3.7]{IY} says:

\begin{theorem}\label{sr}
The canonical functor $\T\to \T/\S$ induces a bijection $\silt_T\T\to \silt\T/\S$ if any object $X$ of $\T$ satisfies $\Hom_\T(T, X[\ell])=0=\Hom_\T(X, T[\ell])$ for $\ell\gg0$.
\end{theorem}

For example, this is the case where $\T$ has a silting object \cite[Proposition 2.4]{AI}.

Now, we are ready to show our main theorem of this section.

\begin{proof}[Proof of Theorem \ref{ns}]
We will apply silting reduction to $\T=\Db(\mod \Lambda)$ and $T=\Lambda$; in this setting, $\S=\thick \Lambda=\Kb(\proj\Lambda)$ and $\T/\S=\ds(\Lambda)$.
To do that, we check that the conditions $\Hom_\T(\Lambda, X[\ell])=0=\Hom_\T(X, \Lambda[\ell])$ are satisfied.
The first equality holds evidently.
Let us show that the second equality holds true.
Since $\Lambda$ has finite right selfinjective dimension,
it can be regarded as a complex in $\Kb(\inj\Lambda)$, which is obtained by applying the Nakayama functor $\nu:=-\otimes_\Lambda^{\bf L}D\Lambda$ to some complex $P$ in $\Kb(\proj\Lambda)$.
Then we get isomorphisms
\[\Hom_\T(X, \Lambda[\ell])\simeq \Hom_{\Kb(\mod\Lambda)}(X, \nu P[\ell])\simeq D\Hom_{\Kb(\mod\Lambda)}(P[\ell], X).\]
As the complex $X$ is bounded, the last above is zero for enoughly large $\ell$.
Thus, silting reduction brings us a bijection $\silt_\Lambda\Db(\mod\Lambda)\to\silt\ds(\Lambda)$.
It follows from \cite[Example 2.5(1)]{AI} that the LHS of the bijection is $\{\Lambda\}$ if $\Lambda$ has finite global dimension, otherwise empty.
Hence, we conclude that $\ds(\Lambda)$ admits no non-zero silting object.
\end{proof}

An algebra $\Lambda$ is said to be \emph{Iwanaga--Gorenstein} if it has finite right and left selfinjective dimension.
In the case, the singularity category $\ds(\Lambda)$ is triangle equivalent to the stable category $\sCM\Lambda$ of the full subcategory of $\mod\Lambda$ consisting of Cohen--Macaulay modules $M$; i.e. $\Ext_\Lambda^i(M,\Lambda)=0$ for $i>0$.
So, we immediately obtain the following corollary.

\begin{corollary}\label{IG}
$\sCM\Lambda$ has no non-zero silting object if $\Lambda$ is Iwanaga--Gorenstein.
\end{corollary}


\section{The finiteness of support $\tau$-tilting modules}\label{tubular}

Let $\Lambda$ be an algebra.
We say that an object $X$ of $\Kb(\proj\Lambda)$ is \emph{two-term} if the $i$th term of $X$ is zero unless $i=0, -1$.
A \emph{support $\tau$-tilting} module is defined to be the 0th cohomology of a two-term silting object of $\Kb(\proj\Lambda)$. (See \cite{AIR} for details.)
Denote by $\sttilt\Lambda$ the set of isomorphism classes of basic support $\tau$-tilting modules.
We call $\Lambda$ \emph{$\tau$-tilting-finite} if $\sttilt\Lambda$ is a finite set.

In this section, we discuss $\tau$-tilting-finiteness of weakly symmetric algebras of tubular type with non-singular Cartan matrix,
which were completely classified up to Morita equivalence by \cite{BS1} as follow (Figure \ref{Lwst}).

\begin{figure}[ht]
\begin{tabular}{cccc}
 $\begin{array}{c}
A_1(\lambda):\\
{\text {\tiny $(\lambda\in K\setminus\{0,1\})$}}\\
\quad \\
\begin{xy}
(0,0) *[o]+{1}="A", (12,0)*[o]+{2}="B", (24,0)*[o]+{3}="C",
\ar @<2pt> "A";"B"^{\alpha}
\ar @<2pt> "B";"A"^{\gamma}
\ar @<2pt> "B";"C"^{\sigma }
\ar @<2pt> "C";"B"^{\beta}
\end{xy} \\
{\tiny \begin{array}{c}
\alpha\gamma\alpha=\alpha\sigma\beta\\
\beta\gamma\alpha=\lambda\beta\sigma\beta\\
  \gamma\alpha\gamma=\sigma\beta\gamma \\
 \gamma\alpha\sigma=\lambda\sigma\beta\sigma \\
\end{array}}
\end{array}
$ & $\begin{array}{c}
A_2(\lambda):\\
{\text {\tiny $(\lambda\in K\setminus\{0,1\})$}}\\
\quad \\
\begin{xy}
(0,0) *[o]+{1}="A", (10,0)*[o]+ {2}="B",
\ar @(ld,lu) "A";"A"^{\alpha}
\ar @(rd,ru) "B";"B"_{\beta}
\ar @<2pt> "A";"B"^{\sigma}
\ar @<2pt> "B";"A"^{\gamma}
\end{xy} \\
\tiny{ \begin{array}{c}
  \alpha^2=\sigma\gamma \\
 \lambda\beta^2=\gamma\sigma \\
 \gamma\alpha=\beta\gamma \\
 \sigma\beta=\alpha\sigma \\
\end{array}}
\end{array}
$ & $ \begin{array}{c}
A_3: \\
\begin{xy}
(-3,-5)*[o]+{1}="A", (6,0)*[o]+{2}="B", (6,9)*[o]+{3}="C", (15,-5)*[o]+{4}="D",
\ar @<2pt>"A";"B"^{\alpha}
\ar @<2pt>"B";"A"^{\beta}
\ar @<2pt>"B";"C"^{\delta}
\ar @<2pt>"C";"B"^{\gamma}
\ar @<2pt>"B";"D"^{\epsilon}
\ar @<2pt>"D";"B"^{\xi}
\end{xy} \\
{\tiny \begin{array}{c}
  \beta\alpha+\delta\gamma+\epsilon\xi=0 \\
  \alpha\beta=0 \\
  \gamma\delta=0 \\
 \xi\epsilon=0 \\
 \end{array}}
 \end{array}
$ & $\begin{array}{c}
A_4: \\
\begin{xy}
(-3,-5)*[o]+{1}="A", (6,0)*[o]+{2}="B", (6,9)*[o]+{3}="C", (15,-5)*[o]+{4}="D",
\ar @<2pt>"A";"B"^{\alpha}
\ar @<2pt>"B";"A"^{\beta}
\ar @<2pt>"B";"C"^{\delta}
\ar @<2pt>"C";"B"^{\gamma}
\ar @<2pt>"B";"D"^{\epsilon}
\ar @<2pt>"D";"B"^{\xi}
\end{xy} \\
{\tiny \begin{array}{c}
  \beta\alpha+\delta\gamma+\epsilon\xi=0 \\
  \alpha\beta=0 \\
  \gamma\epsilon=0 \\
 \xi\delta=0 \\
 \end{array}}
 \end{array}
$
\\
\quad \\
$ \begin{array}{c}
A_5: \\
\begin{xy}
(0,0) *[o]+{1}="A", (10,0) *[o]+{2}="B",
\ar @(ld,lu) "A";"A"^{\alpha}
\ar @<2pt> "A";"B"^{\gamma}
\ar @<2pt> "B";"A"^{\beta}
\end{xy} \\
 \tiny{ \begin{array}{c}
  \alpha^2=\gamma\beta \\
 \beta\alpha\gamma=0 \\
 \quad \\
 \quad  \quad \\
 \quad \quad \\
 \quad \quad \\
 \quad \quad \\
 \quad \quad \\

\end{array}}
\end{array}
$ & $
\begin{array}{c}
A_6: \\
\begin{xy}
(0,0) *[o]+{1}="A", (10,0) *[o]+{2}="B",
\ar @(ld,lu) "A";"A"^{\alpha}
\ar @<2pt> "A";"B"^{\gamma}
\ar @<2pt> "B";"A"^{\beta}
\end{xy} \\
 \tiny{ \begin{array}{c}
  \alpha^3=\gamma\beta \\
 \beta\gamma=0 \\
 \beta\alpha^2=0 \\
 \alpha^2\gamma=0 \\
  \quad \\
 \quad \quad \\
 \quad \quad \\
 \quad \quad \\
\end{array}}
\end{array}
$ & $
\begin{array}{c}
A_7: \\
\begin{xy}
(0,0) *[o]+{1}="A", (10,0) *[o]+{2}="B", (20,0) *[o]+{3}="C", (30,0) *[o]+{4}="D",
\ar @<2pt> "A";"B"^{\alpha}
\ar @<2pt> "B";"A"^{\beta}
\ar @<2pt> "B";"C"^{\delta}
\ar @<2pt> "C";"B"^{\gamma}
\ar @<2pt> "C";"D"^{\epsilon}
\ar @<2pt> "D";"C"^{\xi}
\end{xy} \\
{\tiny \begin{array}{c}
  \beta\alpha=\delta\gamma \\
 \gamma\delta=\epsilon\xi \\
 \alpha\delta\epsilon=0 \\
 \zeta\gamma\beta=0 \\
  \quad \\
 \quad \quad \\
 \quad \quad \\
 \quad \quad \\
\end{array}}
\end{array}
$ & $
\begin{array}{c}
A_8: \\
\begin{xy}
(0,5) *[o]+{1}="A", (0,-5) *[o]+{2}="B", (10,-5) *[o]+{3}="C", (10,5) *[o]+{4}="D",
\ar "A";"B"_{\sigma}
\ar "B";"C"_{\xi}
\ar "C";"D"_{\gamma}
\ar "D";"A"_{\delta}
\ar @<2pt> "A";"C"^{\alpha}
\ar @<2pt> "C";"A"^{\beta}
\end{xy} \\
{\tiny \begin{array}{c}
  \alpha\beta\alpha=\sigma\xi \\
  \beta\alpha\beta=\gamma\delta\\
  \xi\beta\alpha= \delta\alpha\beta=0\\
 \beta\alpha\gamma=\alpha\beta\sigma=0 \\
 \xi\gamma=\delta\sigma=0 \\
\end{array}}
\end{array} $
\quad \\
$ \begin{array}{c}
A_9: \\
\begin{xy}
(0,5) *[o]+{1}="A", (0,-5) *[o]+{2}="B", (10,-5) *[o]+{3}="C", (10,5) *[o]+{4}="D",
\ar "A";"B"_{\alpha}
\ar @<-2pt>"B";"C"_{\sigma}
\ar @<-2pt>"C";"B"_{\beta}
\ar @<-2pt> "C";"D"_{\gamma}
\ar @<-2pt> "D";"C"_{\epsilon}
\ar "D";"A"_{\delta}
\end{xy} \\
{\tiny \begin{array}{c}
  \delta\alpha=\epsilon\beta \\
  \gamma\epsilon=\beta\sigma\\
 \alpha\sigma\beta=0 \\
 \epsilon\gamma\delta=0 \\
 \sigma\gamma\epsilon\gamma=0
 \end{array}}
 \end{array}
$ & $
\begin{array}{c}
A_{10}: \\
\begin{xy}
(8,10) *[o]+{1}="A", (8,1) *[o]+{2}="B", (16,-5) *[o]+{3}="C", (0,-5) *[o]+{4}="D",
\ar @<2pt>"A";"B"^{\beta}
\ar @<2pt>"B";"A"^{\alpha}
\ar "B";"C"^{\delta}
\ar "C";"D"^{\gamma}
\ar "D";"B"^{\xi}
\end{xy} \\
{\tiny \begin{array}{c}
  \xi\alpha\beta=\xi\delta\gamma\xi \\
  \alpha\beta\delta=\delta\gamma\xi\delta\\
 \beta\alpha=0 \\
 (\gamma\xi\delta)^2\gamma=0 \\
\end{array}}
\end{array}
$ & $
\begin{array}{c}
A_{11}: \\
\begin{xy}
(0,0) *[o]+{1}="A", (10,0) *[o]+{2}="B", (20,0) *[o]+{3}="C", (30,0) *[o]+{4}="D",
\ar @<2pt> "A";"B"^{\beta}
\ar @<2pt> "B";"A"^{\alpha}
\ar @<2pt> "B";"C"^{\xi }
\ar @<2pt> "C";"B"^{\gamma}
\ar @<2pt> "C";"D"^{\zeta}
\ar @<2pt> "D";"C"^{\delta}
\end{xy} \\
{\tiny \begin{array}{c}
  \gamma\alpha\beta=\gamma\xi\gamma \\
 \alpha\beta\xi=\xi\gamma\xi \\
 \beta\alpha=0 \\
 \delta\gamma=0 \\
 \xi\zeta=0\\
 (\gamma\xi)^2=\zeta\delta
\end{array}}
\end{array}
$ & $
\begin{array}{c}
A_{12}: \\
\begin{xy}
(0,-5)*[o]+{1}="A",(8,7)*[o]+{2}="B", (16,-5)*[o]+{3}="C",
\ar "A";"B"^{\alpha}
\ar "B";"C"^{\gamma}
\ar @<2pt>"C";"A"^{\beta}
\ar @<2pt>"A";"C"^{\delta}
\end{xy} \\
{\tiny \begin{array}{c}
\delta\beta\delta=\alpha\gamma \\
\gamma\beta\alpha=0 \\
\beta(\delta\beta)^3=0\\
\end{array}}
\end{array}
$
\\
\quad \\
$ \begin{array}{c}
A_{13}: \\
\begin{xy}
(0,0) *[o]+{1}="A", (10,0) *[o]+{2}="B", (20,0) *[o]+{3}="C",
\ar @(ul,ur) "B";"B"^{\alpha}
\ar @<2pt> "A";"B"^{\beta}
\ar @<2pt> "B";"A"^{\gamma}
\ar @<2pt> "B";"C"^{\delta }
\ar @<2pt> "C";"B"^{\sigma}
\end{xy} \\
{\tiny \begin{array}{c}
\alpha^2=\gamma\beta\\
\beta\delta = \beta\gamma=0 \\
 \sigma\gamma=\sigma\alpha=0 \\
 \alpha\delta=0 \\
 \alpha^3=\delta\sigma \\
  \quad \\
\end{array}
}\end{array}
$ & $
\begin{array}{c}
A_{14}: \\
\begin{xy}
(0,0) *[o]+{1}="A", (10,0) *[o]+{2}="B", (20,0) *[o]+{3}="C",
\ar @<2pt> "A";"B"^{\alpha}
\ar @<2pt> "B";"A"^{\beta}
\ar @<2pt> "B";"C"^{\delta }
\ar @<2pt> "C";"B"^{\gamma}
\end{xy} \\
{\tiny \begin{array}{c}
 \quad \\
\beta\alpha=(\delta\gamma)^2 \\
\alpha\delta\gamma\delta=0 \\
  \gamma\delta\gamma\beta=0 \\
 \alpha\beta=0 \\
 \quad \\
 \quad\\
 \end{array}
}\end{array}
$ & $
\begin{array}{c}
A_{15}: \\
\begin{xy}
(0,-5) *[o]+{1}="A",(8,7) *[o]+{2}="B", (16,-5) *[o]+{3}="C",
\ar @(lu,ld) "A";"A"_{\alpha}
\ar "A";"B"^{\sigma}
\ar "B";"C"^{\gamma}
\ar @<2pt>"C";"A"^{\beta}
\ar @<2pt>"A";"C"^{\delta}
\end{xy} \\
{\tiny \begin{array}{c}
\gamma\beta\alpha=0 \\
\alpha^2=\delta\beta \\
\beta\delta=\alpha\sigma=0\\
\alpha\delta=\sigma\gamma\\
\end{array}
}  \end{array}
$ & $
\begin{array}{c}
A_{16}: \\
\begin{xy}
(0,-5) *[o]+{1}="A",(8,7) *[o]+{2}="B", (16,-5) *[o]+{3}="C",
\ar @(lu,ld) "A";"A"_{\alpha}
\ar "B";"A"_{\sigma}
\ar "C";"B"_{\gamma}
\ar @<-2pt>"A";"C"_{\beta}
\ar @<-2pt>"C";"A"_{\delta }
\end{xy} \\
{\tiny \begin{array}{c}
\alpha\beta\gamma=0 \\
\alpha^2=\beta\delta \\
\delta\beta=\sigma\alpha=0\\
\delta\alpha=\gamma\sigma\\
\end{array}
}\end{array} $
\end{tabular}
\caption{List of weakly symmetric algebras of tubular type}
\label{Lwst}
\end{figure}

The main theorem of this section is the following.

\begin{theorem}\label{finiteness1}
Any weakly symmetric algebra of tubular type with non-singular Cartan matrix is $\tau$-tilting-finite.
\end{theorem}

In the appendix, we will see the numbers of support $\tau$-tilting modules of $A_i$'s.

\begin{proof}
Note that $A_i$ but $i=3$ is symmetric \cite[Theorem 2]{BS1}.
Observe that the Cartan matrix of $A_i$ has positive definite.
We then apply \cite[Theorem 13]{EJR}
to deduce the conclusion that $A_i$ but $i=3$ is $\tau$-tilting-finite.
The algebra $A_3$ is just the preprojective algebra of Dynkin type $\mathbb{D}_4$, and so it is $\tau$-tilting-finite by \cite[Theorem 2.21]{M}.
\end{proof}

A selfinjective algebra is said to be \emph{tilting-discrete} if for any $n>0$, there are only finitely many tilting objects of length $n$.
Here is a corollary of Theorem \ref{finiteness1}.

\begin{corollary}
Any weakly symmetric algebra of tubular type with non-singular Cartan matrix is tilting-discrete.
\end{corollary}
\begin{proof}
A weakly symmetric algebra of tubular type with non-singular Cartan matrix is derived equivalent to one of $A_i$'s \cite{BHS},
which is $\tau$-tilting-finite by Theorem \ref{finiteness1}.
It follows from \cite[Corollary 2.11]{AM} that the algebra is tilting-discrete.
\end{proof}

Thanks to Bialkowski--Skowronski \cite{BS2}, we also have a complete list of Morita equivalence classes of selfinjective algebras which are socle equivalent to selfinjective algebras of tubular type.
We focus on such algebras which are not of tubular type.
The following classes of algebras coincide \cite{S}:
\begin{enumerate}[(i)]
\item selfinjective algebras which are socle equivalent to selfinjective algebras of tubular type but not of tubular type;
\item non-standard selfinjective algebras which are socle equivalent to selfinjective algebras of tubular type;
\item non-standard non-domestic selfinjective algebras of polynomial growth;
\item algebras presented by the quivers and relations as in Figure \ref{LLs}.
\begin{figure}[ht]
\begin{center}
\begin{tabular}{cccc}
$ \begin{array}{c}
\Lambda_1: \\
\begin{xy}
(0,0) *[o]+{1}="A", (10,0) *[o]+{2}="B",
\ar @(ld,lu) "A";"A"^{\alpha}
\ar @<2pt> "A";"B"^{\gamma}
\ar @<2pt> "B";"A"^{\beta}
\end{xy} \\
 \tiny{ \begin{array}{c}
  \alpha^2=\gamma\beta \\
 \beta\alpha\gamma=\beta\alpha^2\gamma \\
 \beta\alpha\gamma\beta=0 \\
 \gamma\beta\alpha\gamma=0 \\
\end{array}}
\end{array}
$
&
$\begin{array}{c}
\Lambda_2: \\
\begin{xy}
(0,0) *[o]+{1}="A", (10,0) *[o]+{2}="B",
\ar @(ld,lu) "A";"A"^{\alpha}
\ar @<2pt> "A";"B"^{\gamma}
\ar @<2pt> "B";"A"^{\beta}
\end{xy} \\
 \tiny{ \begin{array}{c}
  \alpha^2\gamma=\beta\alpha^2=0 \\
\gamma\beta\gamma=\beta\gamma\beta=0\\
 \beta\alpha\gamma=\beta\gamma \\
 \alpha^3=\beta\gamma \\
\end{array}}
\end{array}
$
&
$\begin{array}{c}
\Lambda_3(\lambda): \\
{\text {\tiny $(\lambda\in K\setminus\{0,1\})$}}\\
\quad \\
\begin{xy}
(0,0) *[o]+{1}="A", (10,0) *[o]+{2}="B",
\ar @(ld,lu) "A";"A"^{\alpha}
\ar @<2pt> "A";"B"^{\sigma}
\ar @<2pt> "B";"A"^{\gamma}
\ar @(ru,rd) "B";"B"^{\beta}
\end{xy} \\
 \tiny{ \begin{array}{c}
  \alpha^4=\gamma\alpha^2=\alpha^2\sigma=0 \\
\alpha^2=\sigma\gamma+\alpha^3, \lambda\beta^2=\gamma\sigma\\
\gamma\alpha=\beta\gamma, \sigma\beta=\alpha\sigma \\
\end{array}}
\end{array}
$
&
$
\begin{array}{c}
\Lambda_4: \\
\begin{xy}
(0,-5)*[o]+{1}="A",(8,7)*[o]+{2}="B", (16,-5)*[o]+{3}="C",
\ar "A";"B"^{\alpha}
\ar "B";"C"^{\gamma}
\ar @<2pt>"C";"A"^{\beta}
\ar @<2pt>"A";"C"^{\delta}
\end{xy} \\
{\tiny \begin{array}{c}
\delta\beta\delta=\alpha\gamma, (\beta\delta)^3\beta=0 \\
\gamma\beta\alpha\gamma=\alpha\gamma\beta\alpha=0 \\
\gamma\beta\alpha=\gamma\beta\delta\beta\alpha\\
\end{array}}
\end{array}
$
\\
\quad
\\
$ \begin{array}{c}
\Lambda_5: \\
\begin{xy}
(0,0) *[o]+{1}="A", (10,0) *[o]+{2}="B", (20,0) *[o]+{3}="C",
\ar @(ul,ur) "B";"B"^{\alpha}
\ar @<2pt> "A";"B"^{\beta}
\ar @<2pt> "B";"A"^{\gamma}
\ar @<2pt> "B";"C"^{\delta }
\ar @<2pt> "C";"B"^{\sigma}
\end{xy} \\
{\tiny \begin{array}{c}
\alpha^2=\gamma\beta, \alpha^3=\delta\sigma,\\
\sigma\gamma=\alpha\delta=\sigma\alpha=0 \\
\beta\delta=\sigma\gamma=\sigma\alpha=0 \\
\gamma\beta\gamma=\beta\gamma\beta=0\\
\beta\gamma=\beta\alpha\gamma \\
\end{array}
}\end{array}
$
&
$
\begin{array}{c}
\Lambda_6: \\
\begin{xy}
(0,0) *[o]+{1}="A", (10,0) *[o]+{2}="B", (20,0) *[o]+{3}="C",
\ar @<2pt> "A";"B"^{\alpha}
\ar @<2pt> "B";"A"^{\beta}
\ar @<2pt> "B";"C"^{\delta }
\ar @<2pt> "C";"B"^{\gamma}
\end{xy} \\
{\tiny \begin{array}{c}
 \quad \\
\alpha\delta\gamma\delta=\gamma\delta\gamma\beta=0 \\
\alpha\beta\alpha=\beta\alpha\beta=0 \\
\alpha\beta=\alpha\delta\gamma\beta \\
\beta\alpha=\delta\gamma\delta\gamma \\
 \end{array}
}\end{array}
$
&
$
\begin{array}{c}
\Lambda_7: \\
\begin{xy}
(0,-5) *[o]+{1}="A",(8,7) *[o]+{2}="B", (16,-5) *[o]+{3}="C",
\ar @(lu,ld) "A";"A"_{\alpha}
\ar "A";"B"^{\sigma}
\ar "B";"C"^{\gamma}
\ar @<2pt>"C";"A"^{\beta}
\ar @<2pt>"A";"C"^{\delta}
\end{xy} \\
{\tiny \begin{array}{c}
\beta\delta=\beta\alpha\delta, \alpha\sigma=0, \alpha\delta=\sigma\gamma \\
\gamma\beta\alpha=0, \alpha^2=\delta\beta, \gamma\beta\delta=0 \\
\beta\delta\beta=\delta\beta\delta=0\\
\end{array}
}  \end{array}
$
&
$
\begin{array}{c}
\Lambda_8: \\
\begin{xy}
(0,-5) *[o]+{1}="A",(8,7) *[o]+{2}="B", (16,-5) *[o]+{3}="C",
\ar @(lu,ld) "A";"A"_{\alpha}
\ar "B";"A"_{\sigma}
\ar "C";"B"_{\gamma}
\ar @<-2pt>"A";"C"_{\beta}
\ar @<-2pt>"C";"A"_{\delta }
\end{xy} \\
{\tiny \begin{array}{c}
\delta\beta=\delta\alpha\beta, \sigma\alpha=0, \delta\alpha=\gamma\sigma \\
\alpha\beta\gamma=0, \alpha^2=\beta\delta, \delta\beta\gamma=0\\
\beta\delta\beta=\delta\beta\delta=0\\
\end{array}
}\end{array}
$
\\
\quad
\\
&
$ \begin{array}{c}
\Lambda_9: \\
\begin{xy}
(-3,-5)*[o]+{1}="A", (6,0)*[o]+{2}="B", (6,9)*[o]+{3}="C", (15,-5)*[o]+{4}="D",
\ar @<2pt>"A";"B"^{\alpha}
\ar @<2pt>"B";"A"^{\beta}
\ar @<2pt>"B";"C"^{\delta}
\ar @<2pt>"C";"B"^{\gamma}
\ar @<2pt>"B";"D"^{\epsilon}
\ar @<2pt>"D";"B"^{\xi}
\end{xy} \\
{\tiny \begin{array}{c}
  \beta\alpha+\delta\gamma+\epsilon\xi=0 \\
 \gamma\delta=\xi\epsilon=\alpha\beta\alpha=0 \\
  \beta\alpha\beta=0, \alpha\beta=\alpha\delta\gamma\beta \\
   \end{array}}
 \end{array}
$
&
$ \begin{array}{c}
\Lambda_{10}: \\
\begin{xy}
(0,0)*[o]+{1}="A", (15,9)*[o]+{2}="B", (15,0)*[o]+{3}="C",(15,-9)*[o]+{4}="D", (30,0)*[o]+{5}="E",
\ar @<2pt>"A";"B"^{\eta}
\ar @<2pt>"B";"E"^{\mu}
\ar @<-2pt>"A";"C"_{\xi}
\ar @<-2pt>"C";"A"_{\gamma}
\ar @<-2pt>"C";"E"_{\sigma}
\ar @<-2pt>"E";"C"_{\delta}
\ar @<2pt>"E";"D"^{\beta}
\ar @<2pt>"D";"A"^{\alpha}
\end{xy} \\
{\tiny \begin{array}{c}
 \mu\beta=0, \alpha\eta=0, \beta\alpha=\delta\gamma \\
\xi\sigma=\eta\mu, \sigma\delta=\gamma\xi+\sigma\delta\sigma\delta \\
\delta\sigma\delta\sigma=\xi\gamma\xi\gamma=0\\
   \end{array}}
 \end{array}
$ &
\end{tabular}
\end{center}
\caption{List of $\Lambda_i$'s}
\label{LLs}
\end{figure}
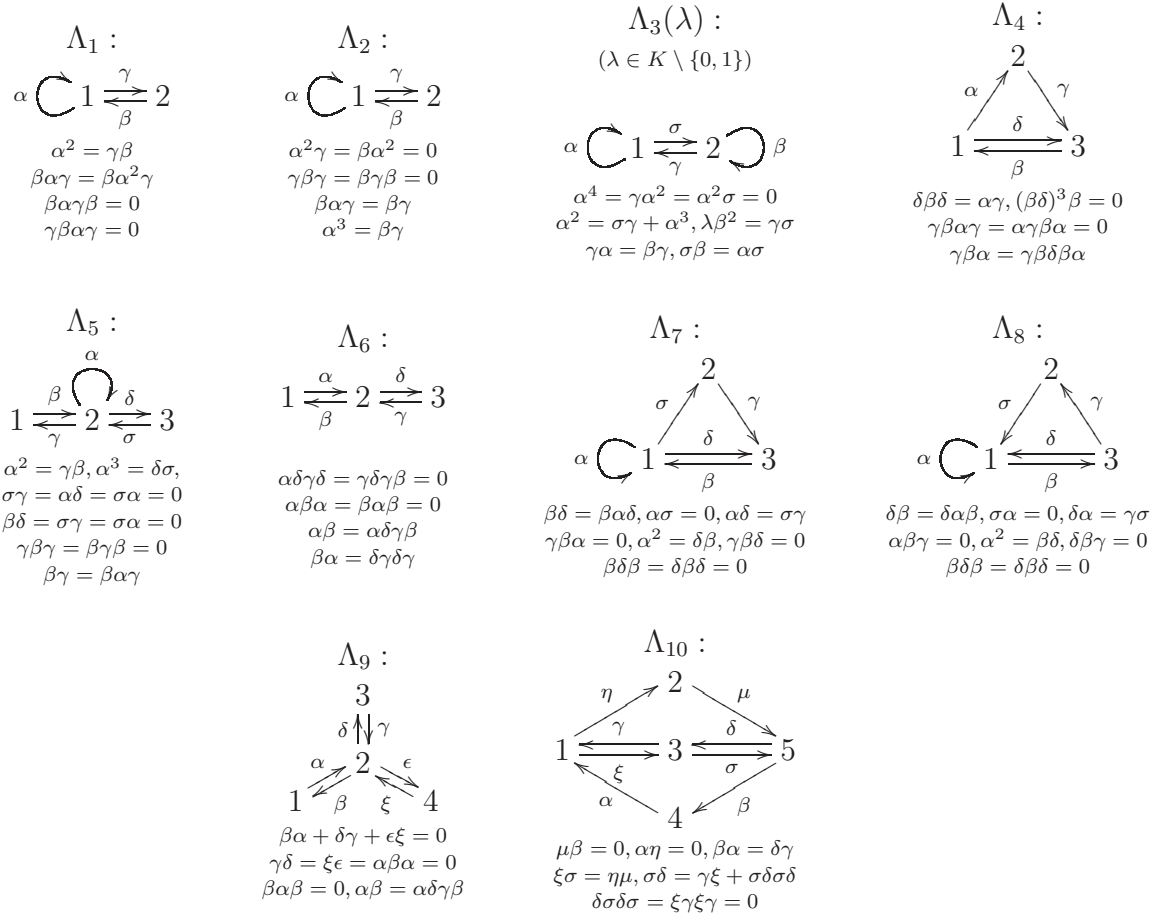
\end{enumerate}
Then we have a similar result as Theorem \ref{finiteness1}.

\begin{theorem}\label{finiteness2}
All $\Lambda_i$ are $\tau$-tilting-finite.
Moreover, they are tilting-discrete.
\end{theorem}

The proof of this theorem is by direct calculation,
and we will give the numbers of support $\tau$-tilting modules of $\Lambda_i$ in the appendix.
So, we now leave here.


\section{Representation-finiteness vs. $\tau$-tilting-finiteness}

The aim of this section is to provide several classes of algebras whose $\tau$-tilting-finiteness implies representation-finiteness.
We start with the following proposition, which was given in \cite{Mo} as a remark; see also \cite{Ad1}.

\begin{proposition}\label{wpc}
A $\tau$-tilting-finite algebra with preprojective or preinjective component is representation-finite.
\end{proposition}

We give several classes of algebras as in Proposition \ref{wpc}.

A \emph{quasitilted} algebra is defined to be the endomorphism algebra of a tilting object $T$ over a hereditary abelian $K$-category $\H$.
When $\H=\mod K\Delta$ for some acyclic quiver $\Delta$,
the algebra is called \emph{tilted} of type $\Delta$.
If in additional, $T$ is preprojective, then the algebra is said to be \emph{concealed}.
We know from \cite{CH} that every quasitilted algebra admits a preprojective component.
This leads to the following corollary, which is a slight generalization of Zito's result \cite[Theorem 3.1]{Z}.

\begin{corollary}\label{Zito}
A $\tau$-tilting-finite quasitilted algebra is representation-finite.
\end{corollary}

Let $\Lambda$ be an algebra associated to an acyclic quiver $Q$ and $i$ a vertex of $Q$.
We write the full subquiver of $Q$ generated by the non-predecessors of $i$ by $Q(i)$.
An algebra $\Lambda$ is said to \emph{satisfy the separation condition}
if for any vertex $i$ of $Q$, all distinct indecomposable summands of $\rad P_i$ have supports lying in different connected components of $Q(i)$.
Here, $P_i$ denotes the indecomposable projective module corresponding to $i$.
In the case, $\Lambda$ admits a preprojective component \cite[IX, Theorem 4.5]{ASS}.
So, we get the following corollary.

\begin{corollary}\label{separation}
A $\tau$-tilting-fintie algebra satisfying the separation condition is representation-finite.
\end{corollary}

Since every tree quiver algebra satisfies the separation condition \cite[IX, Lemma 4.3]{ASS}, the following is also obtained.

\begin{corollary}\label{tree}
A $\tau$-tilting-finite tree quiver algebra is representation-finite.
\end{corollary}

We study a nice (isomorphism) class $\C$ of algebras which are representation-finite or have a tame concealed algebra as a factor.
Such a class contains the classes of algebras with a preprojective component \cite[XIV, Theorem 3.1]{SS}, cycle-finite algebras \cite{MS} and loop-finite algebras \cite[Theorem 4.5]{S2}.
Here is a generalization of Proposition \ref{wpc}.

\begin{proposition}\label{TFRF}
A $\tau$-tilting-finite algebra in $\C$ is representation-finite.
\end{proposition}
\begin{proof}
Combine Corollary \ref{Zito} and \cite[Theorem 5.12(d)]{DIRRT}.
\end{proof}

A \emph{commutative ladder} of degree $n$ is an algebra presented by the quiver
\[\xymatrix{
1 \ar[r]\ar[d] & 2 \ar[r]\ar[d] & \cdots \ar[r]& n \ar[d] \\
1' \ar[r] & 2' \ar[r] & \cdots \ar[r] & n'
}\]
with all possible commutative relations,
which is isomorphic to $K\overrightarrow{\mathbb{A}_2}\otimes K\overrightarrow{\mathbb{A}_n}$.
Here, $\overrightarrow{\mathbb{A}_n}$ stands for the linearly oriented quiver of type $\mathbb{A}_n$.
By \cite[Theorem 3]{EH}, a commutative ladder of degree $n$ is representation-finite if and only if $n\leq4$.
We derive a corollary from Proposition \ref{TFRF}.

\begin{corollary}\label{cl}
A $\tau$-tilting-finite commutative ladder is representation-finite.
\end{corollary}
\begin{proof}
Let $\Lambda$ be a commutative ladder of degree 5.
As the Happel--Vossieck list \cite{HV} (see also \cite{R}), the factor algebra of $\Lambda$ by the idempotents corresponding to the vertices 1 and $5'$ is a tame concealed algebra of type $\widetilde{\mathbb{E}_7}$.
Observe that a commutative ladder of degree $\geq5$ has $\Lambda$ as a factor.
Thus the class of commutative ladders is contained in $\C$.
\end{proof}

We can also deduce Corollary \ref{cl} from Corollary \ref{separation}; this is because a commutative ladder satisfies the separation condition, since all indecomposable projectives have indecomposable radicals.

\begin{remark}\label{Wang}
Inspired by this work, the fourth named author of this paper showed that any $\tau$-tilting-finite strongly simply-connected algebra is representation-finite \cite[Theorem 2.6]{W},
which generalizes Corollary \ref{tree} and \ref{cl}.
\end{remark}

Let us discuss algebras with radical square zero.
To do that, we first recall the definition of separated quivers.

For a quiver $Q$, we construct a new quiver $Q^s$ as follows:
\begin{itemize}
\item the vertices of $Q^s$ are those of $Q$ and their copies;
we donote by $i'$ the copy of a vertex $i$ of $Q$.
\item an arrow $a\to b$ of $Q^s$ are drawn whenever $a$ is a vertex $i$ of $Q$, $b$ is the copy of a vertex $j$ of $Q$, and $Q$ has an arrow $i\to j$.
\end{itemize}
We call the acyclic quiver $Q^s$ the \emph{separated quiver} of $Q$.

As is well-known, a radical square zero algebra presented by a quiver $Q$ is stable equivalent to the hereditary algebra $KQ^s$ \cite[X, Theorem 2.4]{ARS}.
Moreover, we can exactly understand $\tau$-tilting-finite algebras with radical square zero from the shape of $Q^s$ \cite{Ad1}.

Thanks to these results, we show the following result.

\begin{theorem}\label{rs0}
Let $\Lambda$ be an algebra presented by a tree quiver with radical square zero.
\begin{enumerate}
\item If $\Lambda$ is $\tau$-tilting-finite, then it is representation-finite.
\item If the trivial extension of $\Lambda$ is $\tau$-tilting-finite, then it is representation-finite.
\end{enumerate}
\end{theorem}
\begin{proof}
(1) This is due to Corollary \ref{tree}, but we give another proof here, in which we use combinatorial discussion.

As the quiver of $\Lambda$ is tree, we observe that every connected component $R$ of the separated quiver has no same latter $i$ and $i'$.
Then we can apply \cite[Theorem 3.1]{Ad1} for $R$ to deduce the fact that $R$ is of Dynkin type, since $\Lambda$ is $\tau$-tilting-finite.
Hence, it follows from \cite[X, Theorem 2.6]{ARS} that $\Lambda$ is representation-finite.

(2) If the trivial extension $T(\Lambda)$ of $\Lambda$ is $\tau$-tilting-finite, then so is $\Lambda$ by \cite[Theorem 5.12(d)]{DIRRT}, and hence $\Lambda$ is representation-finite by (1).
We observe that $\Lambda$ is simply-connected and has the quadratic form of positive definite, which implies that it is an iterated tilted algebra of Dynkin type \cite[Proposition 5.1]{AS}. (See also \cite{H}.)
It follows from \cite[Theorem 3.1]{AHR} that $T(\Lambda)$ is representation-finite.
\end{proof}

Theorem \ref{rs0} does not necessarily hold if $\Lambda$ is given by a non-tree acyclic quiver.

\begin{example}\label{EXrs0}
\begin{enumerate}
\item Let $\Lambda$ be an algebra presented by the quiver
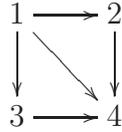
\[\xymatrix{
1 \ar[r]\ar[dr]\ar[d] & 2 \ar[d] \\
3 \ar[r] & 4
}\]
with radical square zero.
Then the separated quiver is the following:
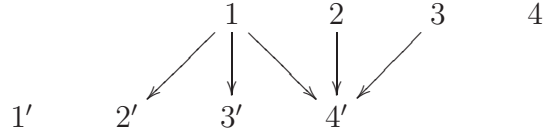
\[\xymatrix{
& & 1 \ar[dl]\ar[d]\ar[dr] & 2 \ar[d] & 3 \ar[dl] & 4 \\
1' & 2' & 3' & 4'
}\]
Observe that it contains the extended Dynkin diagram $\widetilde{\mathbb{D}_5}$ as an underlying graph,
whence $\Lambda$ is $\tau$-tilting-finite by \cite{Ad1} but not representation-finite by \cite{ARS}.

\item Let us consider the algebra presented by the quiver
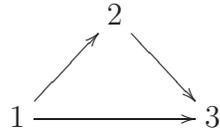
\[\xymatrix{
 & 2 \ar[dr] & \\
1 \ar[rr]\ar[ru] &&3
}\]
with radical square zero.
Then the trivial extension is the Brauer graph algebra given by the Brauer graph
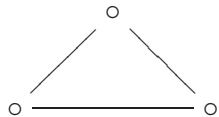
\[\xymatrix{
 & \circ \ar@{-}[dr] & \\
\circ \ar@{-}[rr]\ar@{-}[ru] && \circ
}\]
This is $\tau$-tilting-finite by \cite[Theorem 6.7]{AAC} but not representation-finite.
\end{enumerate}
\end{example}

Let $Q$ be a quiver.
The \emph{double quiver} of $Q$, denoted by $Q^d$, is constructed from $Q$ by adding the inverse arrow of every arrow in $Q$.
Here is an easy observation.

\begin{proposition}\label{dq}
Let $Q$ be a tree quiver and $\Lambda$ an algebra presented by the double quiver of $Q$ with relations.
If $\Lambda$ is $\tau$-tilting-finite, then $Q$ is of Dynkin type.
\end{proposition}
\begin{proof}
By assumtion, it follows from \cite{DIRRT} that $\Lambda/\rad^2\Lambda$ is $\tau$-tilting-finite.
We observe that the separated quiver of $Q^d$ is the disjoint union of two quivers $R_1$ and $R_2$ which satisfy $i\in R_j\Leftrightarrow i'\not\in R_j$ ($j=1,2$) and whose underlying graphs coincide with that of $Q$.
We apply \cite{Ad1} to deduce the fact that $R_1, R_2$, and hence $Q$, are of Dynkin type.
\end{proof}

Let us discuss the locally hereditary case.
An algebra is said to be \emph{locally hereditary} provided every homomorphism between indecomposable projective modules is a monomorphism or zero.
We know that such an algebra is presented by an acyclic quiver and the relations contain no monomials.
We show the following theorem.

\begin{theorem}\label{lh}
A $\tau$-tilting-finite locally hereditary algebra is representation-finite.
\end{theorem}

\begin{proof}
Let $\Lambda$ be a $\tau$-tilting-finite locally hereditary algebra.
As is well-known, the local hereditariness yields that $\Lambda$ has no monomial relation and the quiver $Q$ is acyclic.
The $\tau$-tilting-finiteness implies that $Q$ does not contain a subquiver of extended Dynkin type, whence $\Lambda$ admits all possible commutative relations.
Then, we figure out that $\Lambda$ is strongly simply-connected; see \cite{L} for example.
The assertion follows from \cite[Theorem 2.6]{W}.
\end{proof}

We close this section by giving an interesting observation.
Denote by $\A$ the class of algebras in which $\tau$-tilting-finiteness implies representation-finiteness;
we put a hierarchy of classes contained in $\A$:
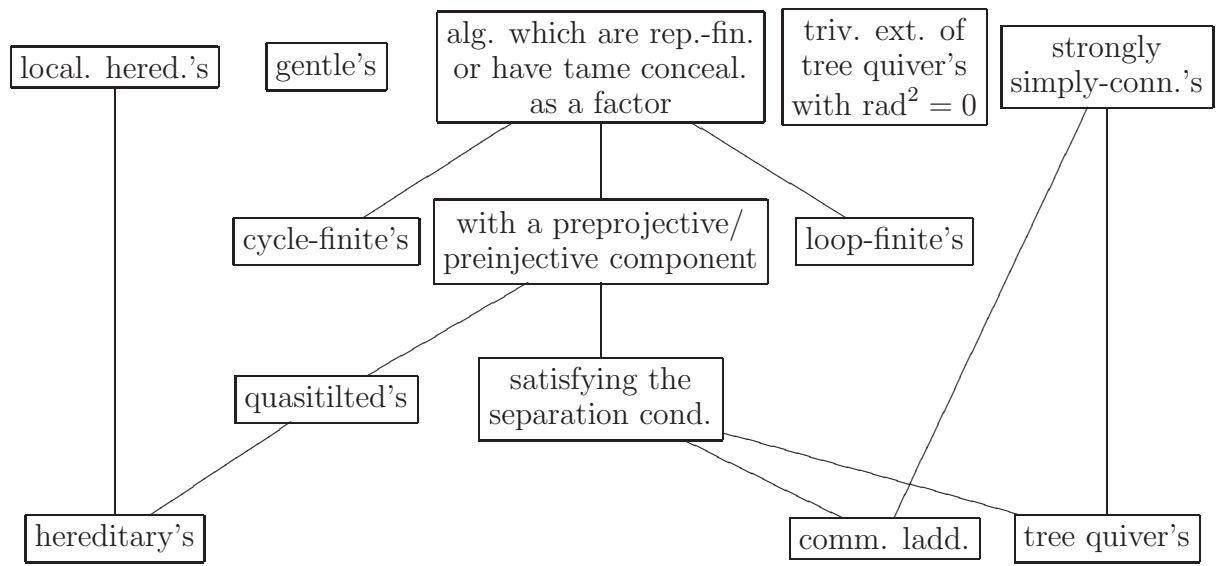
\begin{center}
\[\xymatrix@R=1cm@C=0.2cm{
*+[F]\txt{local. hered.'s}&*+[F]\txt{gentle's}&*+[F]\txt{alg. which are rep.-fin.\\ or have tame conceal.\\ as a factor}\ar@{-}[dr]\ar@{-}[dl]\ar@{-}[d]&*+[F]\txt{triv. ext. of\\ tree quiver's\\ with $\rad^2=0$}&*+[F]\txt{strongly\\ simply-conn.'s}\\
&*+[F]\txt{cycle-finite's}&*+[F]\txt{with a preprojective/\\ preinjective component}\ar@{-}[d]\ar@{-}[ld]&*+[F]\txt{loop-finite's}&\\
&*+[F]\txt{quasitilted's}\ar@{-}[ld]&*+[F]\txt<3cm>{satisfying the separation cond.}\ar@{-}[dr]\ar@{-}[drr]&&\\
*+[F]\txt{hereditary's}\ar@{-}[uuu]&&&*+[F]\txt{comm. ladd.}\ar@{-}[uuur]&*+[F]\txt{tree quiver's}\ar@{-}[uuu]
}\]
\end{center}

\begin{proposition}
The class $\A$ is closed under taking factors by ideals contained in the center and the radical.
\end{proposition}
\begin{proof}
Let $\Lambda$ be in $\A$ and put $\Gamma:=\Lambda/I$, where $I$ is an ideal of $\Lambda$ contained in the center and the radical.
By \cite[Theorem 11]{EJR} (in the appendix), these algebras have the same poset of support $\tau$-tilting modules.
Therefore, if $\Gamma$ is $\tau$-tilting-finite,
then so is $\Lambda$.
By assumption,  it turns out that $\Lambda$ is representation-finite, so is $\Gamma$.
\end{proof}


\appendix
\section{The numbers of support $\tau$-tilting modules over weakly symmetric algebras of tubular type}

In this appendix, we give the numbers of support $\tau$-tilting modules of $A_i$ and $\Lambda_i$ as in Section \ref{tubular}. (See the introduction for the tables of the numbers.)

The following theorem plays an important role.

\begin{theorem}\cite[Theorem 11]{EJR}\label{reduction}
Let $I$ be a two-sided ideal of $\Lambda$ which is contained in the center and the radical of $\Lambda$.
Then we have an isomorphism of posets $\sttilt\Lambda$ and $\sttilt\Lambda/I$.
\end{theorem}

For our algebra $\Lambda$, the strategy is the following.
\begin{enumerate}[(i)]
\item Find central elements which are in the radical.
\item Construct an ideal $I$ generated by the elements as in (i).
\item Consider the factor algebra $\Lambda/I$.
By Theorem \ref{reduction}, we have an isomorphism of posets $\sttilt\Lambda$ and $\sttilt\Lambda/I$.
Then, one counts the number or draws the Hasse quiver of $\sttilt\Lambda/I$.
If possible, we may find a nice algebra whose factor algebra is isomorphic to $\Lambda/I$ and which admits a well-known Hasse quiver of support $\tau$-tilting modules.
\end{enumerate}

\subsection{The number of $\sttilt A_i$}
First, let us discuss for $A_i$'s.
In any case, we can easily check that the following elements belong to the center.
\begin{center}
\begin{tabular}{l@{\hspace{1cm}}l}
$i=1$: $\alpha\gamma+\gamma\alpha$ and $\beta\sigma+\sigma\beta$; & 

$i=2$: $\alpha+\beta$;\\ 

$i=3$: --;&

$i=4$: $\beta\alpha-\gamma\delta-\xi\varepsilon$ and $\alpha\delta\gamma\beta$; \\ 

$i=5$: $\alpha\beta+\beta\alpha$; & 

$i=6$: $\alpha^2$ and $\beta\alpha\gamma$; \\ 

$i=7$: $\alpha\beta+\beta\alpha+\gamma\delta+\xi\varepsilon$;&

$i=8$: $\alpha\beta+\beta\alpha$; \\

$i=9$: $\beta\sigma+\varepsilon\gamma+\sigma\beta$; &

$i=10$: $\alpha\beta+\gamma\xi\delta+\xi\delta\gamma$;\\

$i=11$: $\alpha\beta+\gamma\xi$; &

$i=12$: $\alpha\gamma\beta+\beta\alpha\gamma$ and $\gamma\beta\delta\beta\alpha$; \\

$i=13$: $\alpha^2, \sigma\delta$ and $\beta\alpha\gamma$;&

$i=14$: $\alpha\delta\gamma\beta, \delta\gamma\beta\alpha, \gamma\beta\alpha\delta$ and $\beta\alpha+\gamma\delta\gamma\delta$; \\

$i=15$: $\alpha^2, \beta\alpha\delta$ and $\gamma\beta\sigma$;&

$i=16$: $\alpha^2, \delta\alpha\beta$ and $\sigma\alpha\beta$.
\end{tabular}
\end{center}
Let $I_i$ be the ideal of $A_i$ generated by the elements above and the socle, and $\overline{A_i} :=A_i/I_i$.

In the following, we feel free to utilize Theorem \ref{reduction} and refer to \cite{M} for support $\tau$-tilting modules over preprojective algebras of Dynkin type.

\subsubsection*{$i=1$}

It is seen that $\overline{A_1}$ is isomorphic to the factor algebra of the preprojective algebra of Dynkin type $\mathbb{A}_3$ by the intersection of the center and the radical.
This implies that $A_1$ has 24 support $\tau$-tilting modules.

\subsubsection*{$i=2$}

Observe that $\overline{A_2}$ is the Nakayama algebra presented by the quiver $\xymatrix{\bullet \ar@<2pt>[r]^x & \bullet \ar@<2pt>[l]^y}$ with relations $xy=0=yx$, whence there are 6 support $\tau$-tilting modules of $A_2$.

\subsubsection*{$i=3$}

$A_3$ is the preprojective algebra of type $\mathbb{D}_4$, which has 192 support $\tau$-tilting modules.

\subsubsection*{$i=5,6$}

It is obvious that $\overline{A_5}$ and $\overline{A_6}$ are isomorphic, which are furthermore isomorphic to $R(2AB)$ in Table 2 of \cite{EJR}.
Hence, $A_5$ and $A_6$ have 8 support $\tau$-tilting modules.

\subsubsection*{$i=7$}

By Theorem \ref{reduction}, we have an isomorphism of posets $\sttilt A_7\simeq \sttilt\overline{A_7}$.
Moreover, one observes that $\overline{A_7}$ is isomorphic to the factor algebra of the preprojective algebra $\Gamma$ of type $\mathbb{A}_4$ by the central elements in the radical, and the socle.
However, the socle of $\Gamma$ is not contained in the center, and so we can not apply Theorem \ref{reduction} to obtain the Hasse quiver of support $\tau$-tilting modules.

Now, let us apply Adachi's method \cite{Ad}.
We fix the numbering of the vertices of $\mathbb{A}_4$ by $\xymatrix{1 \ar@{-}[r] & 2 \ar@{-}[r] & 3 \ar@{-}[r] & 4}$ and
let $\overline{\Gamma}$ be the factor algebra of $\Gamma$ by the central elements in the radical.
We can still apply Theorem \ref{reduction} to get an isomorphism $\sttilt\Gamma\simeq \sttilt\overline{\Gamma}$.
Let $P$ be the indecomposable projective module of $\overline{\Gamma}$ corresponding to the vertex 1 and
define a subset $\mathcal{N}$ of $\sttilt\overline{\Gamma}$ by
\[\mathcal{N}:=\{N\in\sttilt\left(\overline{\Gamma}/\soc P\right)\ |\ P/\soc P\in \add N\ \mbox{and } \Hom_{\overline{\Gamma}}(N, P)=0 \}.\]
Here, $\soc P$ stands for the socle of $P$.
We see that $\mathcal{N}$ has 6 elements; see \cite{M} for example.
It follows from \cite[Theorem 3.3(1)]{Ad} that the Hasse quiver of $\sttilt\overline{\Gamma}$ can be constructed by $\sttilt\left(\overline{\Gamma}/\soc P\right)$ and the copy of $\mathcal{N}$.
A similar argument works for the indecomposable projective module $P'$ of $\overline{\Gamma}$ at the vertex 4 instead of $P$.
As $\overline{A_7}$ is isomorphic to the factor algebra of $\overline{\Gamma}$ by the socle of $P$ and $P'$, it turns out that $\overline{A_7}$ has precisely 12 support $\tau$-tilting modules fewer than $\overline{\Gamma}$, so than $\Gamma$.
Consequently, we obtaine that $A_7$ has 108 supoort $\tau$-tilting modules.

\subsubsection*{$i=8,9,11$}

We can use `String Applet' (\url{https://www.math.uni.-bielefeld.de/~jgeuenich/string-applet/}); apply it to $\overline{A_i}$.

\begin{remark}
The applet can be also run for $A_7$.
\end{remark}

\subsubsection*{$i=4, 10$}

We count the number of $\tau$-tilting modules over the factor algebra by each idempotent.
Let $\{e_1,\cdots,e_n\}$ be a complete set of primitive orthogonal idempotents of an algebra $\Lambda$ and $I$ be a subset of $\{1,\cdots,n\}$ (possibly, $I=\emptyset$).
We denote by $t_I$ the number of $\tau$-tilting modules of $\Lambda/(e)$, where $e=\sum_{i\in I}e_i$.
Here, $t_\emptyset$ means the number of $\tau$-tilting modules of $\Lambda$.
Note that the number of support $\tau$-tilting modules over $\Lambda$ is equal to $\sum_I t_I$.

We demonstrate the way of counting for $i=4$; it similarly works for $i=10$.
Putting $\Lambda:=\overline{A_4}$,
$e_i$ denotes the primitive idempotent corresponding to the vertex $i$.
\begin{enumerate}[(i)]
\item We observe that $\Lambda/(e_1)$ is the factor algebra of the Brauer tree algebra of the Brauer tree $\xymatrix{\circ \ar@{-}[r] & \circ \ar@{-}[r] & \circ \ar@{-}[r] & \circ}$ by some socles, and so one easily obtains $t_{\{1\}}=9$.
\item When $I$ has the vertex 2, $\Lambda/(e)$ is semisimple, so $t_I=1$; there are 8 cases.
\item In the cases that $I=\{3\}$ and $\{4\}$, $\Lambda/(e)$ is the preprojective algebra of type $\mathbb{A}_3$, so $t_I=13$; see \cite{M} for example.
\item For $I=\{1,3\}, \{1,4\}, \{3,4\}$, see the case of $i=2$; $t_I=3$.
\item We easily get $t_{\{1,3,4\}}=1$.
\end{enumerate}
We remain to count the number of $\tau$-tilting modules of $\Lambda$.
To do that, we use the GAP-package QPA, and then obtain $t_\emptyset=79$.
Consequently, one sees that there are 132 support $\tau$-tilting modules of $\Lambda$, so of $A_4$.

We only put the table for $A_{10}$.
\[\begin{array}{c||c|c|c|c|c|c|c||c}
I & \begin{subarray}{c}\mbox{{\scriptsize 4 points}}\\\mbox{{\scriptsize 3 points}}\end{subarray}  & \begin{subarray}{c}\{1,2\}\\\{1,3\}\\\{1,4\}\\\{2\}\end{subarray} & \begin{subarray}{c}\{2,3\}\\\{2,4\}\end{subarray} & \{3,4\} & \{1\}  & \begin{subarray}{c}\{3\}\\\{4\}\end{subarray} & \emptyset & \mbox{total}  \\\hline
t_I & 1\ \mbox{(5 cases)} & 2 & 1 & 3 & 10 & 8 & 72 & 116
\end{array}\]

\begin{remark}
It is not difficult to draw the Hasse quivers directly, but they are too large.
\end{remark}

\subsubsection*{$i=12-15$}

We directly construct the Hasse quiver of $\sttilt A_i$ as follows.
\begin{itemize}
\item The Hasse quiver of $\sttilt A_{12}$:
\begin{center}
$\xymatrix@C=1.2cm@R=0.2cm{
&\bullet\ar[r]\ar[drr]&\bullet\ar[r]\ar[dddr]&\bullet\ar[r]\ar[ddrr]&\bullet\ar[r]\ar[dr]&\bullet\ar[r]&\bullet\ar[r]&\bullet\ar[dddr]&\\
&&&\bullet\ar[dr]&&\bullet\ar[r]\ar[ur]&\bullet\ar[ddr]&&\\
&&\bullet\ar[r]\ar[ddr]&\bullet\ar[ddrr]\ar[ddr]&\bullet\ar[r]\ar[ddrr]&\bullet\ar[ur]&&&\\
\bullet\ar[r]\ar[uuur]\ar[dddr]&\bullet\ar[rr]\ar[ur]&&\bullet\ar[uuurr]&&&&\bullet\ar[r]&\bullet\\
&&&\bullet\ar[r]&\bullet\ar[ddr]&\bullet\ar[uuuurr]\ar[ddr]&\bullet\ar[ur]&&\\
&&\bullet\ar[rr]\ar[uuuur]&&\bullet\ar[r]&\bullet\ar[ur]\ar[drr]&&&\\
&\bullet\ar[r]\ar[ur]&\bullet\ar[r]\ar[uur]&\bullet\ar[ur]\ar[rr]&&\bullet\ar[r]&\bullet\ar[r]&\bullet\ar[uuur]&
}$
\end{center}

\item The Hasse quiver of $\sttilt A_{13}$:
\begin{center}
$\xymatrix@C=1.2cm@R=0.3cm{
&\bullet\ar[r]\ar[dddr]&\bullet\ar[r]\ar[drr]&\bullet\ar[rr]\ar[drr]&&\bullet\ar[r]&\bullet\ar[dddr]&\\
&&\bullet\ar[r]\ar[dddr]&\bullet\ar[urr]\ar[ddr]&\bullet\ar[d]&\bullet\ar[ddr]&&\\
&&&&\bullet\ar[ur]\ar[dddr]&&&\\
\bullet\ar[r]\ar[uuur]\ar[dddr]&\bullet\ar[uur]\ar[ddr]&\bullet\ar[r]&\bullet\ar[uur]\ar[ddr]&\bullet\ar[r]&\bullet\ar[uuur]\ar[dddr]&\bullet\ar[r]&\bullet\\
&&&\bullet\ar[d]&&&&\\
&&\bullet\ar[ur]\ar[drr]&\bullet\ar[uur]\ar[drr]&\bullet\ar[r]&\bullet\ar[uur]&&\\
&\bullet\ar[r]\ar[uuur]&\bullet\ar[rr]\ar[urr]&&\bullet\ar[r]&\bullet\ar[r]&\bullet\ar[uuur]&\\
}$
\end{center}

\item The Hasse quiver of $\sttilt A_{14}$:
\begin{center}
$\xymatrix@C=0.9cm@R=0.3cm{
&\bullet\ar[r]\ar[dddr]&\bullet\ar[r]\ar[dr]&\bullet\ar[rrrrr]&&&&&\bullet\ar[r]&\bullet\ar[dddr]&\\
&&&\bullet\ar[rrrr]\ar[drr]&&&&\bullet\ar[ur]\ar[dr]&&&\\
&&\bullet\ar[rr]\ar[uur]&&\bullet\ar[drrr]&\bullet\ar[rr]&&\bullet\ar[r]&\bullet\ar[dr]&&\\
\bullet\ar[uuur]\ar[r]\ar[dddr]&\bullet\ar[ur]\ar[dr]&\bullet\ar[r]&\bullet\ar[urr]\ar[drrr]&&&&\bullet\ar[r]&\bullet\ar[uuur]\ar[dddr]&\bullet\ar[r]&\bullet\\
&&\bullet\ar[r]\ar[dr]&\bullet\ar[rr]\ar[uur]&&\bullet\ar[urr]\ar[drr]&\bullet\ar[uur]\ar[rr]&&\bullet\ar[ur]&&\\
&&&\bullet\ar[rrrr]&&&&\bullet\ar[dr]&&&\\
&\bullet\ar[r]\ar[uuur]&\bullet\ar[ur]\ar[rrrrr]&&&&&\bullet\ar[r]\ar[uur]&\bullet\ar[r]&\bullet\ar[uuur]&\\
}$
\end{center}

\item The Hasse quiver of $\sttilt A_{15}$:
\begin{center}
$\xymatrix@C=1.2cm@R=0.3cm{
&\bullet\ar[r]\ar[dddr]&\bullet\ar[r]\ar[dr]&\bullet\ar[r]\ar[dr]&\bullet\ar[rr]\ar[drr]&&\bullet\ar[r]&\bullet\ar[dddr]&\\
&&\bullet\ar[r]\ar[drr]&\bullet\ar[urrr]&\bullet\ar[r]&\bullet\ar[r]&\bullet\ar[ddr]&&\\
&&&&\bullet\ar[dr]&&&&\\
\bullet\ar[r]\ar[uuur]\ar[dddr]&\bullet\ar[uur]\ar[ddr]&\bullet\ar[r]&\bullet\ar[uur]\ar[dr]&&\bullet\ar[r]&\bullet\ar[uuur]\ar[dddr]&\bullet\ar[r]&\bullet\\
&&&&\bullet\ar[uuur]\ar[drr]&&&&\\
&&\bullet\ar[r]\ar[drr]&\bullet\ar[r]\ar[uuur]&\bullet\ar[uur]\ar[dr]&\bullet\ar[r]\ar[dr]&\bullet\ar[uur]&&\\
&\bullet\ar[r]\ar[uuur]&\bullet\ar[rr]\ar[urrr]&&\bullet\ar[r]&\bullet\ar[r]&\bullet\ar[r]&\bullet\ar[uuur]&\\
}$
\end{center}
\end{itemize}

\begin{remark}
Note that $A_{16}$ is the opposite algebra of $A_{15}$.
So, it follows from \cite[Theorem 2.14]{AIR} that there is a bijection between $\sttilt A_{15}$ and $\sttilt A_{16}$.
We also remark that $\overline{A_{15}}$ is not representation-finite. \end{remark}

\subsection{The number of $\sttilt\Lambda_i$}

Next, we discuss for $\Lambda_i$'s.
One gets central elements.
\begin{center}
\begin{tabular}{l@{\hspace{1cm}}l}
$i=1$: $\alpha^2+\beta\gamma$ and $\beta\alpha\gamma$;&

$i=2$: $\alpha^2, \beta\gamma$ and $\gamma\beta$;\\

$i=3$: $\alpha+\beta, \sigma\gamma$ and $\gamma\sigma$;&

$i=4$: $\gamma\beta\alpha$ and $\beta\alpha\gamma+\alpha\gamma\beta$;\\

$i=5$: $\beta\gamma,\gamma\beta,\delta\sigma$ and $\sigma\delta$;&

$i=6$: $\alpha\beta$ and $\beta\alpha+\gamma\delta\gamma\delta$;\\

$i=7$: $\beta\delta,\delta\beta$ and $\gamma\beta\sigma$;&

$i=8$: --;\\

$i=9$: $\alpha\beta,\gamma\beta\alpha\delta,\delta\gamma\beta\alpha$ and $\xi\delta\gamma\epsilon$;&

$i=10$: $\gamma\xi-\sigma\delta$.
\end{tabular}
\end{center}
Let $I_i$ be the ideal of $\Lambda_i$ generated by the elements above and the socle.
Putting $\overline{\Lambda_i} :=\Lambda_i/I_i$,
we observe isomorphisms as follows.
\[\def\arraystretch{1.5}
\begin{array}{|@{\hspace{7pt}}c@{\hspace{7pt}}|@{\hspace{7pt}}c@{\hspace{7pt}}|@{\hspace{7pt}}c@{\hspace{7pt}}|@{\hspace{7pt}}c@{\hspace{7pt}}|@{\hspace{7pt}}c@{\hspace{7pt}}|@{\hspace{7pt}}c@{\hspace{7pt}}|@{\hspace{7pt}}c@{\hspace{7pt}}|@{\hspace{7pt}}c@{\hspace{7pt}}|@{\hspace{7pt}}c@{\hspace{7pt}}|@{\hspace{7pt}}c@{\hspace{7pt}}|}\hline
 &\overline{\Lambda_1}&\overline{\Lambda_2}& \overline{\Lambda_3} & \overline{\Lambda_4} & \overline{\Lambda_5} & \overline{\Lambda_6} & \overline{\Lambda_7} & \Lambda_8 &\overline{\Lambda_9} \\\hline
\simeq&\overline{A_5}&\overline{A_5}&\overline{A_2}&\overline{A_{12}}&\overline{A_{13}}&\overline{A_{14}}&\overline{A_{15}}&\Lambda_7^\op&\overline{A_3}\\\hline
\end{array}
\]
Here, $\Lambda^\op$ stands for the opposite algebra of an algebra $\Lambda$.
Thus it turns out that $\Lambda_i$ for every $i$ except $i=10$ is $\tau$-tilting-finite by Theorem \ref{finiteness1}.
Moreover, we have the number of support $\tau$-tilting modules of $\Lambda_i$ as in the introduction.

The left case $\Lambda_{10}$ will be dealt with in the forthcoming paper by the first and second named authors.

\end{document}